\documentclass[11pt,reqno]{amsart}
\usepackage{amsmath}
\usepackage{amsthm}
\usepackage{amssymb}
\usepackage{enumerate}
\usepackage{mathtools}
\usepackage{amscd}
\usepackage{color}
\usepackage{graphicx}
\usepackage[all, cmtip]{xy}
\usepackage{soul}
\usepackage{esint}
\usepackage{hyperref}
\hypersetup{pdftex,colorlinks=true,allcolors=blue}
\usepackage{hypcap}
\usepackage[titletoc]{appendix}

\usepackage{amsrefs}

\theoremstyle{plain}
\newtheorem{lemma}{Lemma}

\newtheorem{theo}[lemma]{Theorem}

\newtheorem{remark}[lemma]{Remark}
\newtheorem{Def}[lemma]{Definition}
\newtheorem{thm-Intro}{Theorem} 
\newtheorem{cor-Intro}{Corollary} 

\numberwithin{equation}{section}

\newcommand{\sech}{\textup{sech}\,}

\mathtoolsset{showonlyrefs=true}

\begin{document}
\title[Octonionic Brownian windings]{Octonionic Brownian windings}

	\author{Gunhee Cho}
\address{Department of Mathematics\\
	University of California, Santa Barbara\\
	Santa Barbara, CA 93106}
\email{gunhee.cho@math.ucsb.edu}

 \author{Guang Yang}

\address{Department of Mathematics\\
University of Connecticut\\
196 Auditorium Road,
Storrs, CT 06269-3009, USA}

\email{guang.yang@uconn.edu}

\begin{abstract} 
We define and study windings along Brownian paths on octonionic, Euclidean, projective, and hyperbolic spaces which are isometric to $8$-dimensional Riemannian model spaces. In particular, the asymptotic laws of these windings are shown to be Gaussian for flat and spherical geometries while the hyperbolic winding exhibits different long-time behavior.
\end{abstract}

\maketitle

\tableofcontents

\section{Introduction}

A normed division algebra $A$ over a real field $\mathbb{R}$ is a division algebra over $\mathbb{R}$ which is also a normed vector space, with norm $||\cdot ||$ satisfying
\begin{equation*}
||x\cdot y|| = ||x||\cdot||y||,
\end{equation*}
for all $x$ and $y$ in $A$. A theorem of Hurwitz says that every normed united finite-dimensional algebra over $\mathbb{R}$ is isomorphic to one of $\mathbb{R},\mathbb{C},\mathbb{Q}$, or $\mathbb{O}$.

Hence with the norms $||{\cdot} ||$ of those four algebras $A$, one can talk about the spherical part $\frac{x}{||x||}$ of the element $x \in A$. 

One situation where the spherical part works nicely is the angular motion of the path on the punctured complex plane $\mathbb{C}\backslash \{0\}$. We have a nice polar decomposition of any smooth path $\gamma : [0,\infty) \rightarrow \mathbb{C}\backslash \{0\}$
\begin{equation*}
\gamma(t)=|\gamma(t)|\exp \left(i\int_{\gamma[0,t]} \frac{xdy-ydx}{x^2+y^2}  \right), t\geq 0.
\end{equation*}
Thus we call $\alpha=\frac{xdy-ydx}{x^2+y^2}$ the winding form around $0$ since this form provides the angular motion of the path $\gamma$. Based on this representation, one can also take the integral of the winding form along the path of a Brownian motion $\left(B(t) \right)_{t\geq 0}$ that does not start from $0$ which becomes the Brownian winding functional:
\begin{equation*}
\zeta(t)=\int_{B[0,t]}\alpha.
\end{equation*}
This functional has been studied in the literature and we refer the reader to \cite{YM80,RDYM99,WS00,FBJW16}, and references therein for more details. The goal of this paper is to introduce the octonionic winding form in $8$-dimensional Riemannian homogeneous spaces equipped with an octonionic structure and study Brownian paths.

Studying the winding process in the octonionic setting is indeed important from several perspectives. Of course, the winding process in the complex and quaternionic settings could be extended to the last possible octonionic setting. Although defining the winding process in an octonionic setting is a natural analog from the previous setting, obtaining practical calculations based on Yor’s methods \cite{YM80} is indeed non-trivial, especially for $\mathbb{O}H^1$. Moreover, the winding number process is closely related to the stochastic area process. In particular, in the case of the Heisenberg group, which is a flat-model space of the sub-Riemannian geometry, one can find that the Markov process consists of a Brownian motion and an area process on a horizontal distribution corresponding to the canonical sub-Laplacian. Under the complex field and the quaternion division algebra, the same results could be found by explicit calculation of the sub-Laplacian in the case of the Hopf fibration and the anti-de Sitter space; which are curved model spaces of sub-Riemannian geometry. However, in this explicit calculation, the fact that under the complex field and the quaternion division algebra, both the Hopf fibration and the anti-de Sitter space have contact structures and the fiber has a Lie-group structure can be found \cite{FBJW16,MR4161862}. In contrast, in the case of the octionic setting, structures crucially used for calculation in the complex field and the quaternionic setting are lost. But one can still obtain the radial part of the sub-Laplacian by using the symmetries of octonion \cite{FBGC19,FBGC20}. Therefore, in order to find out the relationship between the sub-Laplacian and the Markov process, a new calculation and verification method is required that does not depend on the structure given only to the complex and quaternionic settings. 

While the octonionic winding process is directly related by definition to the area process, we study the winding process on $8$-dimensional model spaces of Riemannian geometry. This study shows how the octonionic structure can still be used, while also providing a stepping stone for studying the area process and its sub-Riemannian geometry in an octonionic setting.

Our main results are the following: let $\zeta(t)$ be the octonionic Brownian winding functional. Then

\begin{itemize}
	\item on $\mathbb{O}$, the following convergence holds in distribution: 
	\begin{equation*}
	\lim_{t\rightarrow \infty}\frac{\sqrt{6}}{\sqrt{\log t}}\zeta(t)=\mathcal{N}(0,I_7).
	\end{equation*}
	Here, the right-hand side means the normal distribution. 
	\item on $\mathbb{O}P^1$, the following convergence holds in distribution: 
	\begin{equation*}
	\lim_{t\rightarrow \infty}\frac{\zeta(t)}{\sqrt{t}}\rightarrow \mathcal{N}(0,\frac{14}{3} I_7).
	\end{equation*}
	\item on $\mathbb{O}H^1$, we have the following asymptotic result: 
	\begin{equation*}
	\lim_{t\rightarrow \infty} \mathbb{E}(e^{i\lambda\cdot \zeta(t)} )=(\tanh(r(0)))^{\sqrt{9+\mid\lambda\mid^2}-3} \bigg(1+\frac{ 6\sqrt{9+\mid\lambda\mid^2}-18}{\cosh^6(r(0))}A(\lambda) \bigg),
	\end{equation*}
	where \(A(\lambda)\) is given by
	\begin{equation*}
	\frac{\cosh^4(r(0))}{12}+\frac{(\sqrt{9+| \lambda |^2}-2)\cosh^2(r(0))}{60}+\frac{| \lambda |^2-3\sqrt{9+| \lambda |^2}+11}{720}.
	\end{equation*}
\end{itemize}  


\section{Windings of Brownian motion on $\mathbb{O}$}

\subsection{Winding form on $\mathbb{O}$}
The construction principle of Cayley-Dickson algebras can be applied to all real normed division algebras, i.e., the fields of real and complex numbers $\mathbb{R}$ and $\mathbb{C}$, the skew field of Hamiltonian quaternions $\mathbb{H}$ and the non-associative alternative octonions $\mathbb{O}$ as special cases. We refer the reader for instance to \cite{MR1886087,MR4277360} and elsewhere for details and we call the last case the octonions.

The octonions $\mathbb{O}=\left\{x=\sum_{j=0}^{7}x_j e_j, x_j\in \mathbb{R} \right\}$ are an $8$-dimensional algebra with basis $e_0=1, e_1, e_2, e_3, e_4, e_5, e_6, e_7$, and their multiplication is given in the table.
{
\begin{center}
\begin{tabular}{|l|l|l|l|l|l|l|l|}
		\hline
		& $e_1$ & $e_2$ & $e_3$ & $e_4$ & $e_5$ & $e_6$ & $e_7$  \\ \hline
		$e_1$ & $-1$ & $e_4$ & $e_7$ & $-e_2$ & $e_6$ & $-e_5$ & $-e_3$ \\ \hline
		$e_2$ & $-e_4$ & $-1$ & $e_5$ & $e_1$ & $-e_3$ & $e_7$ & $-e_6$ \\ \hline
		$e_3$ & $-e_7$ & $-e_5$ & $-1$ & $e_6$ & $e_2$ & $-e_4$ & $e_1$  \\ \hline
		$e_4$ & $e_2$ & $-e_1$ & $-e_6$ & $-1$ & $e_7$ & $e_3$ & $-e_5$ \\ \hline
		$e_5$ & $-e_6$ & $e_3$ & $-e_2$ & $-e_7$ & $-1$ & $e_1$ & $e_4$  \\ \hline
		$e_6$ & $e_5$ & $-e_7$ & $e_4$ & $-e_3$ & $-e_1$ & $-1$ & $e_2$  \\ \hline
		$e_7$ & $e_3$ & $e_6$  & $-e_1$ & $e_5$  & $-e_4$ & $-e_2$ & $-1$    \\ \hline
\end{tabular}
\end{center}
}

The interesting things from this table are:

\begin{itemize}
	\item $e_1,\cdots,e_7$ are square roots of -1,\\
	\item when $i\neq j$, $e_ie_j=-e_je_i$, \\
	\item the index cycling identity holds: 
	\begin{equation*}
	e_ie_j=e_k \Rightarrow e_{i+1}e_{j+1}=e_{k+1}, 
	\end{equation*}
	where we think of the indices as living in $\mathbb{Z}_7$, and
	\item the index doubling identity holds:
	\begin{equation*}
e_ie_j=e_k \Rightarrow e_{2i}e_{2j}=e_{2k}. 
\end{equation*}	
\end{itemize}
Together with a single nontrivial product like $e_1e_2 = e_4$, these facts recover the whole multiplication table.


The octonionic norm is defined for $x \in \mathbb{O}$ by
\begin{equation*}
||x||^2=\sum_{j=0}^{7}x^2_j.
\end{equation*}
Note that this norm agrees with the standard Euclidean norm on $\mathbb{R}^8$. Also, the existence of a norm on $\mathbb{O}$ implies the existence of inverses for every nonzero element of $\mathbb{O}$.

The set of unit octonions is identified with the $7$-dimensional unit sphere $\mathbb{S}^7\subset \mathbb{O}$  and we identify $\mathbb{O}$ with $\mathbb{R}^8$. Now, we consider $C^1$-path $\gamma : [0,\infty) \rightarrow \mathbb{O}\backslash \{0\}$ and write its polar decomposition:
\begin{equation*}
\gamma(t)=|\gamma(t)|\Theta(t), t\geq 0,
\end{equation*}
with $\Theta(t)\in \mathbb{S}^7 \subset \mathbb{O}$. In particular, the octonionic inverse $\Theta(t)^{-1}$ of $\Theta(t)\in \mathbb{S}^7\subset \mathbb{O}$ is also element of $\mathbb{S}^7$. 

\begin{Def} The winding path $(\theta(t))_{t\geq 0}\in \mathbb{S}^7$ along $\gamma$ is defined by:
\begin{equation}\label{eq:wp}
\theta(t)=\int_{0}^{t}\Theta(s)^{-1}d\Theta(s).
\end{equation}	
The octonionic winding form is the $T_p\mathbb{S}^7$-valued one form $\eta$ such that 
\begin{equation}\label{eq:wf1}
\theta(t)=\int_{\gamma[0,t]}\eta,
\end{equation}
where here, $p$ is the north pole of $\mathbb{S}^7$. 	
\end{Def}
Let us remark that the complex and the quaternionic winding forms considered in \cite{FBJW16, FBJW19} rely on the Lie group structures of $\mathbb{S}^1$ and $\mathbb{S}^3=SU(2)$. In particular, \eqref{eq:wp} could be also written in terms of the Maurer-Cartan form, whereas $\mathbb{S}^7$ does not admit a topological group structure. However, regarding $\mathbb{S}^7$ as a parallelizable manifold, i.e., one that admits global non-vanishing vector fields, we might think of $\eta$ in \eqref{eq:wf1} as a natural generalization of the two previous winding forms having Lie group structures.

Moreover, $\eta$ may be precisely written in octonionic coordinates as 
\begin{equation*}
\eta=\frac{1}{2}\left(\frac{\overline{x}dx-\overline{dx}x }{|x|^2} \right)=\frac{1}{|x|^2} \text{Im}(\overline{x}dx),
\end{equation*}
where $\overline{x}=x_0e_0-\sum_{j=1}^{7}x_ie_i$, $dx=\sum_{j=0}^{7}dx_i e_i$, and $\text{Im}(x)$ is the imaginary part of $x\in \mathbb{O}$. With the octonionic operation, we have $\eta = \sum_{i=1}^{7}\frac{\eta_i}{|x|^2} e_i$, where
\begin{align*}
\eta_1&=(-x_1,x_0,x_3,-x_2,x_5,-x_4,-x_7,x_6)\cdot (dx_0,\cdots,dx_7),\\
\eta_2&=(-x_2,-x_3,x_0,x_1,x_6,x_7,-x_4,-x_5)\cdot (dx_0,\cdots,dx_7),\\
\eta_3&=(-x_3,x_2,-x_1,x_0,x_7,-x_6,x_5,-x_4)\cdot (dx_0,\cdots,dx_7),\\
\eta_4&=(-x_4,-x_5,-x_6,-x_7,x_0,x_1,x_2,x_3)\cdot (dx_0,\cdots,dx_7),\\
\eta_5&=(-x_5,x_4,-x_7,x_6,-x_1,x_0,-x_3,x_2)\cdot (dx_0,\cdots,dx_7),\\
\eta_6&=(-x_6,x_7,x_4,-x_5,-x_2,x_3,x_0,-x_1)\cdot (dx_0,\cdots,dx_7),\\
\eta_7&=(-x_7,-x_6,x_5,x_4,-x_3,-x_2,x_1,x_0)\cdot (dx_0,\cdots,dx_7).
\end{align*}
Here $\cdot$ is the dot product (also see the appendix in \cite{GMMMI11}).

\subsection{Asymptotic winding of octonionic Brownian motion}
We got the following definition from the previous section:
\begin{Def}
	The winding number of an octonionic Brownian motion {$W=\sum_{i=0}^{7}W_i e_i$}, not started from $0$, is defined by the Stratonovitch stochastic line integral:
	\begin{equation*}
	\zeta(t):=\int_{W[0,t]}\eta, t\geq 0.
	\end{equation*}
\end{Def}

From a well-known consequence of the skew-product decomposition of Euclidean Brownian motions, we have the following lemma (see \cite{PEJRLCG88}).

\begin{lemma}
	Let {$W=\sum_{i=0}^{7}W_i e_i$} be an octonionic Brownian motion not started from $0$. There exists a Bessel process $(R(t))_{t\geq 0}$ of dimension $8$ (or equivalently index one) and a $\mathbb{S}^7$-valued Brownian motion $\Theta(t)_{t \geq 0}$ independent from the process $(R(t))_{t\geq 0}$ such that 
\begin{equation*}
W(t)=R(t)\Theta(A_t),
\end{equation*}	
where 
\begin{equation*}
A_t:=\int_{0}^{t}\frac{ds}{R^2(s)}.
\end{equation*}
	
\end{lemma}

\begin{theo}
		Let \(\zeta(t) \) be the winding process of a Brownian motion on \(\mathbb{O} \). Then we have the following asymptotic result
	\begin{equation*}
	\lim_{t\rightarrow \infty}\frac{\sqrt{6}}{\sqrt{\log t}}\zeta(t)=\mathcal{N}(0,I_7).
	\end{equation*}
\end{theo}

\begin{proof}
	 We provide two different proofs. The first one relies on properties of Bessel processes in Euclidean spaces. The second one is Yor's method.

	   \textbf{Proof 1:} By Hartman-Watson's law, for \(\mid W_0 \mid=\rho  \) we have
		\begin{equation*}
		\mathbb{E}_\rho(e^{-\frac{\mid \lambda \mid^2}{2}A_t } \mid R(t)=r )=\frac{I_{\sqrt{9+\mid \lambda \mid^2}}(\frac{\rho r}{t})}{I_3(\frac{\rho r}{t})}
		\end{equation*}
		(for example, see \cite{Y80}), where $I_\alpha$ is the modified Bessel function
		\begin{equation}
			\label{Modified Bessel function}
			I_{\alpha}(x)=\sum_{j=0}^{\infty}\frac{1}{\Gamma(1+j+\alpha ) j!}\left(\frac{x}{2}\right)^{2j+\alpha}.
		\end{equation}
		Recall that the transition density of Bessel process with index \(\delta \) is given by
		\begin{equation*}
			p^\delta_t(x,y)=t^{-1}\left(\frac{y}{x}\right)^\delta \exp\left\{\frac{x^2+y^2}{2t} \right\}I_\delta\left(\frac{xy}{t}\right).
		\end{equation*}
		Now by using the density of the Bessel process with index  \(3\) and performing a change of variables \(r\rightarrow \frac{r}{\sqrt{t}} \), it is readily checked that
		\begin{equation*}
		\mathbb{E}_\rho(e^{-\frac{\mid \lambda \mid^2}{2}A_t })=\frac{e^{-\frac{\rho^2}{2t}}}{\rho^3}\int_{0}^{\infty}r^4e^{-\frac{r^2}{2}}t^{\frac{3}{2}}I_{\sqrt{9+\mid \lambda \mid^2}}\left(\frac{\rho r}{\sqrt{t}}\right)dr.
		\end{equation*}
		By \eqref{Modified Bessel function}, we have 
		\begin{equation*}
		I_{\sqrt{9+\mid \lambda \mid^2}}\left(\frac{\rho r}{\sqrt{t}}\right)=\sum_{j=0}^{\infty}\frac{1}{\Gamma(1+j+\sqrt{9+\mid \lambda \mid^2} ) j!}\left(\frac{\rho r}{2\sqrt{t}}\right)^{2j+\sqrt{9+\mid \lambda \mid^2}}.
		\end{equation*}
		To get the asymptotic behavior, one only needs to work with the first term of the series. Performing the transformation \(\lambda(t)=\sqrt{\frac{6}{\log t}}\lambda \), we have
		\begin{align}
		\lim_{t\rightarrow \infty}& \mathbb{E}_\rho(e^{-\frac{\mid \lambda(t) \mid^2}{2}A_t })\\
		&=\lim_{t\rightarrow \infty} \frac{e^{-\frac{\rho^2}{2t}}}{\rho^3\Gamma(1+\sqrt{9+\mid \lambda(t) \mid^2} ) }\int_{0}^{\infty}e^{-\frac{r^2}{2}}r^4t^{\frac{3}{2}}\left(\frac{\rho r}{2\sqrt{t}}\right)^{\sqrt{9+\mid \lambda(t) \mid^2}}dr.\label{limit} \\
		\end{align}
		Since \(\lim_{t\rightarrow \infty}{\Gamma(1+\sqrt{9+\mid \lambda(t) \mid^2} ) }=\Gamma(4)=6,   \) we see that the right-hand side of \eqref{limit} is equal to
		\begin{equation*}
		\lim_{t\rightarrow \infty} \frac{1}{6\rho^3} \int_{0}^{\infty}t^{\frac{3}{2}}\left(\frac{\rho r}{2\sqrt{t}}\right)^{\sqrt{9+\mid \lambda(t) \mid^2}} e^{-\frac{r^2}{2}}r^4dr=\lim_{t\rightarrow \infty} \frac{1}{48} \int_{0}^{\infty}\left(\frac{\rho r}{\sqrt{t}}\right)^{\sqrt{9+\mid \lambda(t) \mid^2}-3} e^{-\frac{r^2}{2}}r^7dr.
		\end{equation*}
		By L'Hôpital's rule, we have \(\lim_{t\rightarrow \infty}\left(\frac{\rho r}{\sqrt{t}}\right)^{\sqrt{9+\mid \lambda(t) \mid^2}-3}= e^{-\frac{1}{2}\mid \lambda \mid^2}  \). Finally, by using integration by parts repeatedly we have
		\begin{equation*}
			\lim_{t\rightarrow \infty} \mathbb{E}_\rho(e^{-\frac{\mid \lambda(t) \mid^2}{2}A_t })=e^{-\frac{1}{2}\mid \lambda \mid^2}.
		\end{equation*}
		
		\textbf{Proof 2:} First note that, the $8$-dimensional Bessel process solves the following stochastic differential equation
		\begin{equation*}
		dR(t)=\frac{7}{2R(t)}dt+dB_t,\; R(0)=\rho,
		\end{equation*}
		where \(B_t \) is a real-valued standard Brownian motion. We define a local martingale 
		\begin{equation*}
		D^{\mu}_t=\exp\bigg(\mu\int_{0}^{t}\frac{1}{R(s)}dB_s-\frac{\mu^2}{2}\int_{0}^{t}\frac{1}{R^2(s)}ds  \bigg). 
		\end{equation*}
		By It\^o's formula, we have
		\begin{equation*}
		D^{\mu}_t=\bigg(\frac{R(t)}{R(0)}\bigg)^{\mu}\exp\bigg(-\frac{\mu^2+6\mu}{2}\int_{0}^{t}\frac{1}{R^2(s)}ds   \bigg).
		\end{equation*}
		Now we define a new probability measure
		\begin{equation*}
		\mathbb{P}^{\mu}\mid_{\mathcal{F}_t}=D^{\mu}_t\mathbb{P}\mid_{\mathcal{F}_t}.
		\end{equation*}
		Under this new probability, \(R(t) \) is given by
		\begin{equation*}
		dR(t)=\frac{7+2\mu}{2R(t)}dt+dB_t,\; R(0)=\rho.
		\end{equation*}
		Let \(\mu=\sqrt{9+| \lambda |^2}-3 \). Then we have
		\begin{equation*}
		\mathbb{E}_\rho(e^{-\frac{| \lambda |^2}{2}A_t })=\mathbb{E}^{\mu}_\rho  \bigg( \frac{\rho}{R(t)}    \bigg)^\mu.  
		\end{equation*}
		By applying the scaling \(\lambda(t)=\sqrt{\frac{6}{\log t}}\lambda \) and the density of \(R(t)\), one can recover \eqref{limit}, hence the desired result. 

\end{proof}

\section{Windings of Brownian motion on $\mathbb{O}P^1$}
\subsection{Winding form on $\mathbb{O}P^1$}
The  unit sphere in $\mathbb{O}^2$  is given by
\begin{equation*}
\mathbb{S}^{15}=\{(x,y)\in\mathbb{O}^2, ||(x,y)||=1 \}.
\end{equation*}
We have a Riemannian submersion $\pi : \mathbb{S}^{15} \rightarrow \mathbb{O}P^1$, given by $(x,y) \mapsto [x:y]$, where $[x:y]=y^{-1}x$. Then the vertical distribution $\mathcal{V}$ and the horizontal distribution $\mathcal{H}$  of $T\mathbb{S}^{15}$ are defined by $\ker d\pi$ and the orthogonal complement of $\mathcal{V}$ respectively so that $T\mathbb{S}^{15}=\mathcal{H} \oplus \mathcal{V}$. Note that $\pi : \mathbb{S}^{15} \rightarrow \mathbb{O}P^1$ has totally geodesic fibers and for each $b\in \mathbb{O}P^1$, the  fiber ${\pi}^{-1}(\{b\})$ is isometric to $\mathbb{S}^7$ with the standard sphere metric $g_{\mathbb{S}^7}$. 

This submersion $\pi$ yields the octonionic Hopf fibration:
\begin{equation*}
\mathbb{S}^7 \hookrightarrow \mathbb{S}^{15} \rightarrow \mathbb{O}P^1
\end{equation*}
(see \cite{FBGC19}). In particular, viewing the fiber $\mathbb{S}^7$ as the set of unit octonions, one can take the quotient space 
\begin{equation*}
\mathbb{S}^{15}/ \mathbb{S}^7
\end{equation*}
which is the octonionic projective space $\mathbb{O}P^1$.  Note that $\mathbb{O}P^1$ is isometric to the $8$-dimensional Euclidean sphere with radius $\frac{1}{2}$ (see Theorem 3.5 in \cite{Escobales}). 

Since we identify $\mathbb{S}^{15}$ as a subset of $\mathbb{O}^2$ and $\mathbb{S}^7$ as the unit sphere in $\mathbb{O}$, its quotient space $\mathbb{O}P^1$ consists of elements written by the homogeneous coordinate $[x:y]=\left\{(\lambda x, \lambda y) : (x,y)\in \mathbb{S}^{15},  \lambda \in \mathbb{S}^{7}  \right\}$. Thus we can parametrize $\mathbb{O}P^1$ using the octonionic inhomogeneous coordinate:
\begin{equation*}
w={y}^{-1}x, (x,y)\in \mathbb{S}^{15}
\end{equation*}
with the convention $0^{-1}x=\infty$. This coordinate allows us to identify $\mathbb{O}P^1$ with its one-point compactification $\mathbb{O}\cup \left\{\infty \right\}$. With this inhomogeneous coordinate, the Riemannian distance from $0$ is given by 
 \begin{equation}\label{inhomo}
 r=\arctan |w|.
 \end{equation}
 
Let $\gamma : [0,\infty) \rightarrow \mathbb{O}P^1\backslash\left\{0,\infty \right\}$ be a $C^1$-path. We have its polar decomposition:
\begin{equation*}
\gamma(t)=|\gamma(t)|\Theta(t),
\end{equation*}
with $\Theta(t)\in \mathbb{S}^7$ and defining the winding path $\theta(t)\in T_p \mathbb{S}^7$ as 
\begin{equation*}
\theta(t)=\int_{0}^{t}{\Theta(s)}^{-1}d\Theta(s),
\end{equation*}
where $p$ is the north pole of $\mathbb{S}^7$. 

The octonionic winding form on $\mathbb{O}P^1$ is then the $T_p \mathbb{S}^7$-valued one-form $\eta$ such that
\begin{equation*}
\theta(t)=\int_{\gamma[0,t]}\eta=\int_{0}^{t}\frac{\overline{\gamma}(s)d\gamma(s)-d\overline{\gamma}(s)\gamma(s) }{|\gamma(s)|^2}ds, t\geq 0.
\end{equation*}

\subsection{Asymptotic winding of Brownian motion on $\mathbb{O}P^1$}

Let us first consider a general Riemannian manifold $(M^n,g)$. With the geodesic polar coordinate $(r,\theta^1,\cdots,\theta^{n-1})$, the Riemannian metric $g$ is given by 
\begin{equation*}
g=dr^2+\sum_{i,j}g_{ij}d\theta^{i}d\theta^{j}.
\end{equation*}
Then the Laplace-Beltrami operator $\triangle_{M}$ with this coordinate is written as 
\begin{equation*}
\triangle_{M}=\frac{\partial^2}{\partial r^2}+\frac{1}{\sqrt{\det g} }\left(\frac{\partial}{\partial r}\sqrt{\det g} \right)\frac{\partial}{\partial r}+\frac{1}{\sqrt{\det g} }\sum_{i,j=1}^{n-1}\frac{\partial}{\partial \theta^i}\left(g^{ij}\sqrt{\det g}\frac{\partial}{\partial \theta^{j}}\right)\\
\end{equation*}
We can decompose the Laplace-Beltrami operator as:
\begin{equation*}
\triangle_{M}:=\mathcal{L}+\triangle_{\mathbb{S}^{n-1}(r) },
\end{equation*}
where 
\begin{align*}
\mathcal{L}=&\frac{\partial^2}{\partial r^2}+\frac{1}{\sqrt{\det g} }\left(\frac{\partial}{\partial r}\sqrt{\det g} \right)\frac{\partial}{\partial r},\\
\triangle_{\mathbb{S}^{n-1}(r) }&=\frac{1}{\sqrt{\det g} }\sum_{i,j=1}^{n-1}\frac{\partial}{\partial \theta^i}\left(g^{ij}\sqrt{\det g}\frac{\partial}{\partial \theta^{j}}\right).
\end{align*}
Here, $\mathbb{S}^{n-1}(r)$ is the geodesic sphere of radius $r$. $\mathcal{L}$ is called the radial part of the Laplace-Beltrami operator. General formulas for the radial parts of Laplacian-Beltrami operators on rank one symmetric spaces are well-known (Chapter 3 in \cite{MR496885}, also p171 in \cite{SH65} and \cite{MR754767}). Notice that as a Riemannian manifold, $\mathbb{O}P^1$ is a compact rank one symmetric space. Moreover, it is isometric to the 8-dimensional euclidean sphere of sectional curvature $\frac{1}{4}$.

Now, let $(M,g)$ be the $8$-dimensional euclidean sphere of the sectional curvature $\frac{1}{4}$. Then with the inhomogeneous coordinate \eqref{inhomo}, those operators must be written as   
\begin{align}\label{eq:LB-OP1}
&\mathcal{L}=\frac{\partial^2}{\partial r^2}+14\cot(2r)\frac{\partial}{\partial r} ,\\
&\triangle_{\mathbb{S}^{7}(r) }=\frac{2}{\sin^2 2r}\triangle_{\mathbb{S}^7},. \nonumber
\end{align}
Here $\triangle_{\mathbb{S}^7}$ is the Laplace-Beltrami operator on $\mathbb{S}^7$ (\cite{FB02, BAL78,SH65}).

Equivalently, $\triangle_{\mathbb{O}P^1}$ with an inhomogeneous coordinate $w\in \mathbb{O}$ is given by 
\begin{equation*}
\triangle_{\mathbb{O}P^1}=4(1+|w|^2)^2\text{Re}\left(\frac{\partial^2}{\partial \overline{w}\partial w}\right)-24(1+|w|^2)\text{Re}\left( {w\frac{\partial}{\partial w}}\right).
\end{equation*}
Here, with $w=\sum_{i=0}^{7}x_i e_i$,
\begin{equation*}
\frac{\partial}{\partial w}=\frac{1}{2} \left( \frac{\partial}{\partial x_0}e_0-\sum_{i=1}^{7}\frac{\partial }{\partial x_j}e_j\right).
\end{equation*}
In real coordinates, we have
\begin{equation*}
\triangle_{\mathbb{O}P^1}=\sec^4 r \left(\sum_{i=0}^{7}\frac{\partial^2}{\partial {x_i}^2} \right)-12\sec^2 r\left(\sum_{i=0}^{7}x_i\frac{\partial}{\partial x_i} \right).
\end{equation*}

Thus, the Brownian motion $\left(w(t)\right)_{t\geq 0}$ in $\mathbb{O}P^1$ solves the stochastic differential equation:
\begin{equation*}
dw(t)=\sec^2 r(t)dW(t)-6\sec^2 r(t)w(t)dt,
\end{equation*}
where $\tan r(t)=|w(t)|$ and $\left(W(t)\right)_{t\geq 0}$ is a standard Brownian motion in $\mathbb{O}$. Thus we can write the winding process as 
\begin{equation*}
\zeta(t)=\frac{1}{2}\int_{0}^{t}\frac{\overline{w}(s)dw(s)-d\overline{w}(s)w(s)}{|w(s)|^2},
\end{equation*}
or equivalently,
\begin{equation*}
\zeta(t)=\frac{1}{2}\int_{0}^{t}\frac{\overline{w}(s)dW(s)-d\overline{W}(s)w(s)}{\sin^2 r(s) }.
\end{equation*}

As in the flat setting, the study of $\zeta$ makes use of the following skew-product decomposition.

\begin{lemma}
Let $w$ be a Brownian motion on $\mathbb{O}P^1$ not started from $0$ or $\infty$. Then there exists a Jacobi process $(r(t))_{t\geq 0}$ with generator
\begin{equation*}
\frac{1}{2}\left(\frac{\partial^2}{\partial r^2}+14\cot(2r)\frac{\partial}{\partial r} \right)
\end{equation*}
and a Brownian motion $\Theta(t)$ on $\mathbb{S}^7$ independent from the process $(r(t))$ such that 
\begin{equation*}
w(t)=\tan r(t)\Theta_{A_t},
\end{equation*}
where 
\begin{equation*}
A_t:=\int_{0}^{t}\frac{4ds}{\sin^2(2r(s))}.
\end{equation*}
\begin{proof}
This follows from \cite{PEJRLCG88} and with the fact from \eqref{eq:LB-OP1}, $\frac{1}{2}\triangle_{\mathbb{O}P^1}$ may be decomposed in polar coordinates as:
\begin{equation*}
\frac{1}{2}\left(\frac{\partial^2}{\partial r^2}+14\cot 2r \frac{\partial}{\partial r}+\frac{4}{\sin^2 2r}\triangle_{\mathbb{S}^7} \right).
\end{equation*}

\end{proof}

\end{lemma}

\begin{theo}
	We have
\begin{equation*}
\frac{\zeta(t)}{\sqrt{t}}\rightarrow \mathcal{N}(0,\frac{14}{3} I_7),
\end{equation*}	
in distribution.
\end{theo}

\begin{proof}
	Let \(\lambda\in \mathbb{R}^7 \). From the previous lemma, we have
\begin{equation*}
\mathbb{E}(e^{i\lambda \cdot \zeta(t)} )=\mathbb{E}(e^{-\frac{| \lambda |^2}{2}A_t})=e^{-2| \lambda |^2t}\mathbb{E}(e^{-2| \lambda |^2\int_{0}^{t}\cot^22r(s)ds} ).
\end{equation*}
Note that \(r(t)\) solves the following stochastic differential equation
\begin{equation*}
r(t)=r(0)+7\int_{0}^{t}\cot 2r(s) ds+B_t,
\end{equation*}
where \(B_t \) is a real-valued standard Brownian motion. We consider the following local martingale
\begin{equation*}
D^{\mu}_t=\exp\bigg(  2\mu\int_{0}^{t}\cot 2r(s)dB_s-2\mu^2\int_{0}^{t}\cot^2 2r(s)ds       \bigg).  
\end{equation*} 
By It\^o's formula, we have 
\begin{equation*}
D^{\mu}_t=e^{2\mu t}\left(\frac{\sin 2r(t)}{\sin 2r(0)}\right)^\mu\exp \bigg(-2(\mu^2+6\mu)\int_{0}^{t}\cot^22r(s)ds \bigg).
\end{equation*} 
Let \(\mu=\sqrt{| \lambda |^2 +9}-3 \). Then by Girsanov's theorem, we have
\begin{equation*}
\mathbb{E}(e^{i\lambda \cdot \zeta(t)} )=(\sin 2r(0))^{\sqrt{| \lambda |^2 +9}-3} e^{-2t(| \lambda |^2+\sqrt{| \lambda |^2 +9}-3 )}\mathbb{E}^\mu\left( \sin 2r(t)^{3-\sqrt{| \lambda |^2 +9} }\right).
\end{equation*}
Applying the scaling $\lambda\rightarrow \frac{\lambda}{\sqrt{t}}$, we arrive at
\begin{equation*}
\lim_{t\rightarrow \infty}\mathbb{E}(e^{i\lambda \cdot \frac{\zeta(t)}{\sqrt{t} } } )=(\sin 2r(0))^{\sqrt{ \frac{\mid \lambda\mid^2}{t}  +9}-3} e^{-2t\left( \frac{\mid \lambda\mid^2}{t} +\sqrt{ \frac{\mid\lambda\mid^2}{t}  +9}-3 \right)}\mathbb{E}^\mu\left( \sin 2r(t)^{3-\sqrt{ \frac{\mid\lambda\mid^2}{t} +9} }\right)
\end{equation*}
Taking the limit gives
\begin{equation*}
	\lim_{t\rightarrow \infty}\mathbb{E}(e^{i\lambda \cdot \frac{\zeta(t)}{\sqrt{t} } } )=\lim_{t\rightarrow \infty}e^{-2t\left( \frac{\mid\lambda\mid^2}{t} +\sqrt{ \frac{\mid\lambda\mid^2}{t}  +9}-3 \right)}=e^{-\frac{7}{3}| \lambda |^2}
\end{equation*}
\end{proof}

\section{Windings of Brownian motion on $\mathbb{O}H^1$}
\subsection{Winding form on $\mathbb{O}H^1$}

Similarly, the octonionic anti-de Sitter space $AdS^{15}(\mathbb{O})$  is defined as the pseudo-hyperbolic space by:
\begin{equation*}
AdS^{15}(\mathbb{O})=\{(x,y)\in\mathbb{O}^2, ||(x,y)||^2_{\mathbb{O}}=-1 \},
\end{equation*}
where 
\begin{equation*}
||(x,y)||^2_{\mathbb{O}}:=||x||^2-||y||^2.
\end{equation*}

The map  $\pi : AdS^{15}(\mathbb{O}) \rightarrow \mathbb{O}H^1$, given by $(x,y) \mapsto [x:y]=y^{-1}x$  is a pseudo-Riemannian submersion with totally geodesic fibers isometric to $\mathbb S^7$.  The pseudo-Riemannian submersion $\pi$ yields the octonionic anti-de Sitter fibration
\begin{equation*}
\mathbb{S}^7 \hookrightarrow {AdS}^{15}(\mathbb{O}) \rightarrow \mathbb{O}H^1,
\end{equation*}
(see \cite{FBGC20}).
Viewing the fiber $\mathbb{S}^7$ as a set of unit octonions, one can take the quotient space
\begin{equation*}
AdS^{15}(\mathbb{O})/ \mathbb{S}^7.
\end{equation*}
The quotient space is the octonionic hyperbolic space $\mathbb{O}H^1$. As in the $\mathbb{O}P^1$ case, one can parametrize $\mathbb{O}H^1$ using the octonionic inhomogeneous coordinate:
\begin{equation*}
w={y}^{-1}x, (x,y)\in AdS^{15}(\mathbb{O}),
\end{equation*}
with the convention $0^{-1}x=\infty$. This coordinate allows us to identify $\mathbb{O}H^1$ with the open unit ball $\{w\in \mathbb{O} : |w|<1 \}$. With this inhomogeneous coordinate, the Riemannian distance from $0$ is given by 
\begin{equation}\label{inhomo2}
\tanh r= |w|.
\end{equation}

Let $\gamma : [0,\infty) \rightarrow \mathbb{O}P^1\backslash\left\{0,\infty \right\}$ be a $C^1$-path. Then we have its polar decomposition:
\begin{equation*}
\gamma(t)=|\gamma(t)|\Theta(t),
\end{equation*}
with $\Theta(t)\in \mathbb{S}^7$ and defining the winding path $\theta(t)\in T_p \mathbb{S}^7$ as 
\begin{equation*}
\theta(t)=\int_{0}^{t}{\Theta(s)}^{-1}d\Theta(s),
\end{equation*}
where $p$ is the north pole of $\mathbb{S}^7$. The octonionic winding form on $\mathbb{O}H^1$ is then the $T_p \mathbb{S}^7$-valued one-form $\eta$ such that
\begin{equation*}
\theta(t):=\int_{w[0,t]}\eta=\frac{1}{2} \int_{0}^{t}\frac{dw(s)\overline{w}(s)-d\overline{w}(s)w(s) }{|w(s)|^2}.
\end{equation*}

\subsection{Asymptotic winding of Brownian motion on $\mathbb{O}H^1$}



As we did in a previous section, a similar analogue holds for $(M,g)$ in the $8$-dimensional euclidean unit ball as the Riemannian homogeneous space. With the inhomogeneous coordinate \eqref{inhomo2}, once one writes $\triangle_{\mathbb{O}H^1}=\mathcal{L}+\triangle_{\mathbb{S}^{7}(r) }$, then those operators must be written as   
\begin{align}\label{eq:LB-OP2}
&\mathcal{L}=\frac{\partial^2}{\partial r^2}+14\coth(2r)\frac{\partial}{\partial r} ,\\
&\triangle_{\mathbb{S}^{7}(r) }=\frac{2}{\sinh^2 2r}\triangle_{\mathbb{S}^7}, \nonumber
\end{align}
(for example, see \cite{FBGC20,BAL78,SH65}).

Equivalently, $\triangle_{\mathbb{O}H^1}$ with a coordinate $w\in \mathbb{O}$ is given by 
\begin{equation*}
\triangle_{\mathbb{O}P^1}=4(1-|w|^2)^2\text{Re}\left(\frac{\partial^2}{\partial \overline{w}\partial w}\right)+24(1-|w|^2)\text{Re}\left( {w\frac{\partial}{\partial w}}\right),
\end{equation*}
here, with $w=\sum_{i=0}^{7}x_i e_i$,
\begin{equation*}
\frac{\partial}{\partial w}=\frac{1}{2} \left( \frac{\partial}{\partial x_0}e_0-\sum_{i=1}^{7}\frac{\partial }{\partial x_j}e_j\right).
\end{equation*}
In real coordinates, we have
\begin{equation*}
\triangle_{\mathbb{O}H^1}=\sech^4 r \left(\sum_{i=0}^{7}\frac{\partial^2}{\partial {x_i}^2} \right)+12\text{}\sech^2 r\left(\sum_{i=0}^{7}x_i\frac{\partial}{\partial x_i} \right).
\end{equation*}

Thus, the Brownian motion $\left(w(t)\right)_{t\geq 0}$ in $\mathbb{O}H^1$ solves the stochastic differential equation:
\begin{equation*}
dw(t)=\sech^2 r(t)dW(t)+6 \sech^2 r(t)w(t)dt,
\end{equation*}
where $\tanh r(t)=|w(t)|$ and $\left(W(t)\right)_{t\geq 0}$ is a standard Brownian motion in $\mathbb{O}$. Thus we can write the winding process as 
\begin{equation*}
\zeta(t)=\frac{1}{2}\int_{0}^{t}\frac{\overline{w}(s)dw(s)-d\overline{w}(s)w(s)}{|w(s)|^2},
\end{equation*}
or equivalently,
\begin{equation*}
\zeta(t)=\frac{1}{2}\int_{0}^{t}\frac{\overline{w}(s)dW(s)-d\overline{W}(s)w(s)}{\sinh^2 r(s) }.
\end{equation*}

As before, to study $\zeta$, we shall make use of a skew-product decomposition.

\begin{lemma}
	Let $w$ be a Brownian motion on $\mathbb{O}H^1$ not started from $0$. Then there exists a Jacobi process $(r(t))_{t\geq 0}$ with generator
	\begin{equation*}
	\frac{1}{2}\left(\frac{\partial^2}{\partial r^2}+14\coth(2r)\frac{\partial}{\partial r} \right)
	\end{equation*}
	and a Brownian motion $\Theta(t)$ on $\mathbb{S}^7$ independent from the process $(r(t))$ such that 
	\begin{equation*}
	w(t)=\tan r(t)\Theta_{A_t},
	\end{equation*}
	where 
	\begin{equation*}
	A_t:=\int_{0}^{t}\frac{4ds}{\sinh^2(2r(s))}.
	\end{equation*}
	\begin{proof}
This follows from \cite{PEJRLCG88} and with the fact from \eqref{eq:LB-OP1}, $\frac{1}{2}\triangle_{\mathbb{O}P^1}$ may be decomposed in polar coordinates as:
		\begin{equation*}
		\frac{1}{2}\left(\frac{\partial^2}{\partial r^2}+14\coth 2r \frac{\partial}{\partial r}+\frac{4}{\sinh^2 2r}\triangle_{\mathbb{S}^7} \right).
		\end{equation*}
		
	\end{proof}
\end{lemma}

\begin{theo}
	For any \(\lambda\in\mathbb{R}^7 \), we have the following asymptotic result
	\begin{equation*}
	\lim_{t\rightarrow \infty} \mathbb{E}(e^{i\lambda\cdot \zeta(t)} )=(\tanh(r(0)))^{\sqrt{9+| \lambda |^2}-3} \bigg(1+\frac{ 6\sqrt{9+| \lambda |^2}-18}{\cosh^6(r(0))}A(\lambda) \bigg),
	\end{equation*}
	where \(A(\lambda)\) is given by
	\begin{equation*}
	\frac{\cosh^4(r(0))}{12}+\frac{(\sqrt{9+| \lambda |^2}-2)\cosh^2(r(0))}{60}+\frac{| \lambda |^2-3\sqrt{9+| \lambda |^2}+11}{720}.
	\end{equation*}
\end{theo}

\begin{proof}
	Let \(\beta \) be a standard Brownian motion on \(\mathbb{R}^7 \). Since \(\zeta(t)=\beta_{\int_{0}^{t}\frac{4ds}{\sinh^2(2r(s))} } \) in distribution, the characteristic function of \(\zeta(t) \) is given by
	\begin{equation*}
	\mathbb{E}(e^{i\lambda\cdot \zeta(t)} )=\mathbb{E}(e^{-| \lambda |^2\int_{0}^{t}\frac{2ds}{\sinh^2(2r(s))} } ).
	\end{equation*}
	In order to apply Yor's method, we introduce the following local martingale
	\begin{equation*}
	D^{(a,b)}_t=\exp \left\{ \int_{0}^{t}a \coth(r(s))+b \tanh (r(s))dB_s-\frac{1}{2}\int_{0}^{t}(a \coth(r(s))+b \tanh (r(s)) )^2ds   \right\}.
	\end{equation*}
	By It\^{o}'s formula, 
	\begin{align*}
	D^{(a,b)}_t=e^{-4(a+b+\frac{ab}{4})t}& \bigg(\frac{\sinh(r(t) )}{\sinh(r(0))}   \bigg)^a  \bigg(\frac{\cosh(r(t) )}{\cosh(r(0))}   \bigg)^b\\
	&\exp \left\{-\frac{a^2+6a}{2}\int_{0}^{t}\coth^2(r(s))ds-\frac{b^2+6b}{2}\int_{0}^{t}\tanh^2(r(s))ds  \right\}.
	\end{align*}
	Note that
	\begin{equation*}
	\frac{1}{\sinh^2(2r(s))}=\frac{\coth^2(r(s))+\tanh^2(r(s))-2}{4}.
	\end{equation*}
	Let \(\hat{a} ,\hat{b}  \) be the solution of 
	\begin{equation*}
	x^2+6x-| \lambda |^2=0.
	\end{equation*}
	More precisely, 
	\begin{equation*}
	\hat{a}=-3+\sqrt{9+| \lambda |^2}\;,\; \hat{b}=-3-\sqrt{9+| \lambda |^2}.
	\end{equation*}
	Then 
	\begin{align*}
	D^{(\hat{a},\hat{b})}_t&=e^{24t} \bigg(\frac{\sinh(r(t) )}{\sinh(r(0))}   \bigg)^{-3+\sqrt{9+| \lambda |^2}}  \bigg(\frac{\cosh(r(t) )}{\cosh(r(0))}   \bigg)^{-3-\sqrt{9+| \lambda |^2}}\exp \left\{ \int_{0}^{t}\frac{-2| \lambda |^2}{\sinh^2(2r(s))} ds  \right\}\\
	&=e^{24t} \bigg(\frac{\tanh(r(t) )}{\tanh(r(0))}   \bigg)^{-3+\sqrt{9+\mid \lambda \mid^2}}  \bigg(\frac{\cosh(r(0) )}{\cosh(r(t))}   \bigg)^{6}\exp \left\{ \int_{0}^{t}\frac{-2| \lambda |^2}{\sinh^2(2r(s))} ds  \right\}.
	\end{align*}
	Since \(D^{(\hat{a},\hat{b})}_t \) is bounded, it is of class DL and hence a martingale. We can now define a new probability
	\begin{equation*}
	\mathbb{P}^{(\hat{a},\hat{b})}\mid_{\mathcal{F}_t}=D^{(\hat{a},\hat{b})}_t\mathbb{P}\mid_{\mathcal{F}_t}.
	\end{equation*}
	By Girsanov's theorem, we have
	\begin{align*}
	\mathbb{E}(e^{i\lambda\cdot \zeta(t)} )&=\mathbb{E}^{(\hat{a},\hat{b}) }\left( e^{-24t} \bigg(\frac{\tanh(r(t) )}{\tanh(r(0))}   \bigg)^{3-\sqrt{9+| \lambda |^2}}  \bigg(\frac{\cosh(r(t) )}{\cosh(r(0))}   \bigg)^{6}\right) \\
	&=e^{-24t} \frac{(\tanh(r(0)))^{\sqrt{9+| \lambda |^2}-3} }{\cosh^6(r(0))} \mathbb{E}^{(\hat{a},\hat{b}) }\left( \frac{\cosh^6(r(t))}{\tanh(r(t) )^{\sqrt{9+| \lambda |^2}-3} }    \right)  \\
	&=e^{-24t} \frac{(\tanh(r(0)))^{\sqrt{9+| \lambda |^2}-3} }{\cosh^6(r(0))} \mathbb{E}^{(\hat{a},\hat{b}) }\left( (1-\frac{1}{\cosh^2(r(t))} )^{-\frac{\hat{a}}{2}} \cosh^6(r(t))  \right)\\
	&=e^{-24t} \frac{(\tanh(r(0)))^{\sqrt{9+| \lambda |^2}-3} }{\cosh^6(r(0))} \sum_{k=0}^{\infty}\frac{(\hat{a}/2)_k}{k!} \mathbb{E}^{(\hat{a},\hat{b}) }\left(\frac{1}{\cosh^{2k-6}(r(t))}\right). 
	\end{align*}
	As before, the asymptotic behavior is determined by the lowest order term. Thus,
	\begin{equation*}
	\lim_{t\rightarrow \infty} \mathbb{E}(e^{i\lambda\cdot \zeta(t)} )=\frac{(\tanh(r(0)))^{\sqrt{9+| \lambda |^2}-3} }{\cosh^6(r(0))} \lim_{t\rightarrow \infty} e^{-24t}\mathbb{E}^{(\hat{a},\hat{b}) }(\cosh^{6}(r(t))).
	\end{equation*}
	Under this new probability measure \(\mathbb{P}^{(\hat{a},\hat{b})} \), \(r(t) \) solves the following stochastic differential equation
	\begin{equation*}
	r(t)=r(0)+\int_{0}^{t}\bigg[ (\hat{a}+\frac{7}{2})\coth(r(s))+(\hat{b}+\frac{7}{2} )\tanh(r(s))      \bigg]    ds+d\hat{B}_t,
	\end{equation*}
	where \(\hat{B}_t \) is a standard Brownian motion under \(\mathbb{P}^{(\hat{a},\hat{b})}  \). By It\^o's formula, we have
	\begin{align*}
	\frac{d }{dt}\mathbb{E}^{(\hat{a},\hat{b}) }(\cosh^{6}(r(t)))&=(6\hat{a}+6\hat{b}+18)\mathbb{E}^{(\hat{a},\hat{b})}(\cosh^{6}(r(t)))-(6\hat{b}+\frac{5}{2})\mathbb{E}^{(\hat{a},\hat{b})}(\cosh^{4}(r(t)))\\
	&=24\mathbb{E}^{(\hat{a},\hat{b})}(\cosh^{6}(r(t)))-(6\hat{b}+36)\mathbb{E}^{(\hat{a},\hat{b})}(\cosh^{4}(r(t))).
	\end{align*}
	By solving this differential equation, we have
	\begin{equation*}
	\mathbb{E}^{(\hat{a},\hat{b}) }(\cosh^{6}(r(t)))=e^{24t}\bigg(\cosh^{6}(r(0))-(6\hat{b}+36)\int_{0}^{t}\frac{\mathbb{E}^{(\hat{a},\hat{b}) }(\cosh^{4}(r(s)))}{e^{24s}} ds\bigg).
	\end{equation*}
	By using a similar argument, we have
	\begin{align*}
	\mathbb{E}^{(\hat{a},\hat{b}) }(\cosh^{4}(r(t)))&=e^{12t}\bigg(\cosh^{4}(r(0))-(4\hat{b}+20)\int_{0}^{t}\frac{\mathbb{E}^{(\hat{a},\hat{b}) }(\cosh^{2}(r(s)))}{e^{12s}} ds\bigg),\\
	\mathbb{E}^{(\hat{a},\hat{b}) }(\cosh^{2}(r(t)))&=e^{4t}\bigg(\cosh^{2}(r(0))- \frac{2\hat{b}+8}{4}\bigg)+\frac{2\hat{b}+8}{4}.
	\end{align*}
	Putting everything together, we end up with
	\begin{equation*}
	\lim_{t\rightarrow \infty} \mathbb{E}(e^{i\lambda\cdot \zeta(t)} )=\frac{(\tanh(r(0)))^{\sqrt{9+| \lambda |^2}-3} }{\cosh^6(r(0))}\bigg(\cosh^{6}(r(0))+(6\sqrt{9+| \lambda |^2}-18)A(\lambda) \bigg),
	\end{equation*}
	where \(A(\lambda)\) is given by
	\begin{equation*}
	\frac{\cosh^4(r(0))}{12}+\frac{(\sqrt{9+| \lambda |^2}-2)\cosh^2(r(0))}{60}+\frac{| \lambda |^2-3\sqrt{9+| \lambda |^2}+11}{720}.
	\end{equation*}
	
\end{proof}

\begin{remark}
Unlike in the case of $\mathbb{O}$ or $\mathbb{O}P^1$, the formula of the limiting behavior of the Brownian winding functional of $\mathbb{O}H^1$ is different from its counterpart of $4$-dimensional quaternionic hyperbolic space $\mathbb{H}H^1$. This is due to the fact that limit behaviour of the Brownian winding in $\mathbb{H}H^1$ and $\mathbb{O}H^1$ are sensitive to the parameter of the corresponding Bessel process, while in the first two cases, different parameters only give different scaling.    
\end{remark}

{\bf Acknowledgement:} The authors thank the anonymous referee for a careful reading and remarks that greatly improved the presentation of the paper. The second author is supported by National Science Foundation grant DMS-1901315. Both authors thank professor Fabrice Baudoin for the very helpful discussions. 

\bibliographystyle{spmpsci}
\bibliography{reference}
\end{document}